\documentclass[12pt,a4paper]{article}

\usepackage[margin=1in]{geometry} 
\usepackage{amsfonts,amscd,amssymb, mathtools,mathrsfs, dsfont, bbm,bbding} 
\usepackage[amsmath,amsthm,thmmarks]{ntheorem} 
\usepackage{graphicx,xypic,color} 
\usepackage{indentfirst}
\usepackage[colorlinks=true,citecolor=blue]{hyperref}

\newtheorem{thm}{Theorem}[section]
\newtheorem{prop}[thm]{Proposition}

\newtheorem{lem}[thm]{Lemma}
\newtheorem{defn}{Definition}[section]
\newtheorem{assmp}{Assumption}[section]
\newtheorem{remark}{Remark}[section]
\newtheorem{eg}{Example}[section]

\numberwithin{equation}{section} 

\renewcommand{\geq}{\geqslant}
\renewcommand{\leq}{\leqslant}

\newcommand{\citethm}[1]{Theorem \ref{#1}}
\newcommand{\citeprop}[1]{Proposition \ref{#1}}

\newcommand{\citelem}[1]{Lemma \ref{#1}}

\newcommand{\citeassmp}[1]{Assumption \ref{#1}}
\newcommand{\citeremark}[1]{Remark \ref{#1}}

\newcommand{\opfont}{\mathbb}

\newcommand{\BE}[2][]{\ensuremath{\operatorname{\opfont{E}}_{#1}\!\left[#2\right]}}
\newcommand{\bp}{\ensuremath{\opfont{P}}}

\newcommand{\BF}{\ensuremath{\mathcal{F}}}

\newcommand{\R}{\ensuremath{\operatorname{\mathbb{R}}}}

\newcommand{\dd}{\ensuremath{\operatorname{d}\! }}
\newcommand{\dt}{\ensuremath{\operatorname{d}\! t}}
\newcommand{\ds}{\ensuremath{\operatorname{d}\! s}}

\newcommand{\du}{\ensuremath{\operatorname{d}\! u}}

\newcommand{\dx}{\ensuremath{\operatorname{d}\! x}}
\newcommand{\dy}{\ensuremath{\operatorname{d}\! y}}
\newcommand{\dz}{\ensuremath{\operatorname{d}\! z}}
\newcommand{\ddp}{\ensuremath{\operatorname{d}\! p}}

\newcommand{\sets}{\mathcal{S}}

\newcommand{\nn}{\nonumber}
\begin{document}
\title{
$g$-Expectation of Distributions\thanks{The authors thank two anonymous referees for their comments.}}
\author{Mingyu Xu\thanks{School of Mathematical Sciences, Fudan University, Shanghai, P.R. China; xumy@fudan.edu.cn.},
\and Zuo Quan Xu\thanks{Department of Applied Mathematics, The Hong Kong Polytechnic University, Kowloon, Hong Kong, China; maxu@polyu.edu.hk. This author acknowledges financial support from the NSFC (No.11971409), The Hong Kong RGC (GRF No.15202421), The PolyU-SDU Joint Research Center on Financial Mathematics, The CAS AMSS-POLYU Joint Laboratory of Applied Mathematics, and The Hong Kong Polytechnic University.},
\and Xun Yu Zhou\thanks{Department of Industrial Engineering and Operations Research \&
The Data Science Institute,
Columbia University,
New York, NY 10027, USA; xz2574@columbia.edu. This author acknowledges financial support through a start-up grant and the Nie Center for Intelligent Asset Management at Columbia University.}}
\date{\today}
\maketitle

\begin{abstract}
We define $g$-expectation of a distribution as the infimum of the $g$-expectations of all the terminal random variables sharing that distribution. We present two special cases for nonlinear $g$ where the $g$-expectation of distributions can be
explicitly derived. As a related problem, we introduce the notion of law-invariant $g$-expectation and provide its sufficient conditions. Examples of application in financial dynamic portfolio choice are supplied.

\bigskip

\end{abstract}

\section{Introduction}

Backward stochastic differential equation (BSDE) is inherently a ``sample-wise" notion, by which we mean that the terminal value of a BSDE is a {\it random variable} (as opposed to a {\it distribution}) and its solution is in terms of {\it sample paths}
(as opposed to {\it distributions} in the space of continuous functions). This is natural from the very origin of BSDE -- the adjoint equation in the maximum principle for stochastic control whose terminal value is the gradient of the terminal reward function evaluated at the optimal terminal state \cite{B73}. This sample-wise property is vital for applying BSDE to financial derivative pricing.
This is because a BSDE is employed to describe a wealth process that must replicate the final payoff of a derivative security in almost {\it all} states of nature in order to avoid arbitrage.

However, there are applications in which one needs to ``duplicate" only a probability distribution at the end. For example, in a dynamic Merton problem with the objective of maximizing the expected utility of terminal wealth, a common approach, called the martingale method, is to first find the optimal terminal wealth (which is a random variable) and then replicate that wealth via a BSDE. However, as the expected utility is {\it law-invariant}, we do not really care about which particular terminal random variable to be replicated; all we need is to find the optimal terminal {\it distribution} and then replicate a random variable that follows that distribution.

Generalizing this idea to problems where people are concerned with only the distribution of the final state of a stochastic dynamic system motivates us to formulate and study BSDEs with terminal distributions. However, there is an issue of non-uniqueness: given a distribution there may be (infinitely) many associated random variables leading to different solutions and in particular different initial values of the classical BSDEs (called the {\it g-expectations} introduced by Peng \cite{P97}).
A natural way to address this issue is to find the one that gives the {\it minimum} initial value and use that value to define the $g$-expectation of the given {\it distribution}. This idea has actually been around in the finance literature related to {\it linear} BSDEs.
An investment strategy as well as its terminal payoff are called {\it cost efficient} if any other strategy generating the same terminal distribution costs at least as much. In an arbitrage-free complete market setting, Dybvig \cite{D98a,D98b} characterizes the cost-efficient payoffs for investors with increasing and law-invariant preferences. 
Bernard et al. \cite{BBV14} study cost-efficient strategies subject to state-dependent constraints. Wang et al. \cite{WXXY20} identify cost-efficient strategies under multi stochastic dominance constraints.
On the other hand, cost-efficient strategies play a vital role in solving continuous-time portfolio selection problems with non-expected utility preferences especially in behavioral finance settings. The essential idea is that one needs to only consider cost-efficient strategies as candidates for optimal strategies if the risk preference is increasing and law-invariant (which is the case with most preference measures including the behavioral ones; see \cite{HZ11} for many examples). When the wealth equation or the corresponding BSDE is linear, it follows from the celebrated Hardy--Littlewood lemma that cost-efficient strategies must be anti-comonotonic with the pricing kernel and thus have a specific form. This special structure enables one to solve the optimization problem by considering the quantile function of the terminal wealth as the decision variable. This approach, coined as the ``quantile formulation" by \cite{HZ11}, was taken in
Jin and Zhou \cite{JZ08} to solve for the first time a behavioral portfolio selection problem with probability distortion and S-shaped utility functions. It is subsequently employed to solve various non-expected utility portfolio selection problems, such as in \cite{HZ11,X14,XZ16,X16,MX21}.

All the above works in the financial application context have a crucial assumption that the wealth equation is linear (or equivalently the market pricing rule is linear) upon which the Hardy--Littlewood lemma is premised.
In practice, however, the wealth equation may be nonlinear. For instance, if the deposit and loan rates in a market are different, then the resulting wealth equation becomes nonlinear. Indeed, Pardoux and Peng \cite{PP90} use this example to motivate the introduction of {\it nonlinear} BSDEs. The objective of this paper is to
formulate and study the problem of finding the smallest nonlinear $g$-expectation value among terminal random variables following a same given distribution. To our best knowledge, this paper is the first to do so. In addition to its own theoretical interest, solving this problem will in turn help us solve non-expected utility maximization problems (including those in behavioral finance) in general ``nonlinear markets" with nonlinear pricing rules.
%
%

In this paper, we first formulate the problem and define the $g$-expectation of distributions. Then we present two special cases of the driver $g$ where the problem can be solved explicitly, including the benchmark case corresponding to the aforementioned market model where the deposit and loan rates are different.
Finally, we introduce the related notion of law-invariant $g$-expectation, and present a large class of such law-invariant drivers by a partial differential equation (PDE) argument. We also give several financial examples to demonstrate the theoretical results.

The rest of the paper is organized as follows. In Section 2, we formulate the problem and provide its solution in the linear BSDE case. Section 3 is devoted to the study of two special nonlinear cases. In Section 4, we introduce law-invariant $g$-expectation and give sufficient conditions for law-invariant drivers. We conclude in Section 5.

\section{Problem formulation}
Throughout this paper, a filtered probability space $(\Omega,{\cal F},\bp; \{{\cal F}_t\}_{t\geq0})$ is given satisfying the usual assumptions, along with a standard $n$-dimensional Brownian motion $B=\{B_t,\;t\geq0\}$. We assume
$\{{\cal F}_t\}_{t\geq0}$ is generated by $B$ complemented by all the $\bp$-null sets, and ${\cal F}={\cal F}_T$ where $T>0$ is a given maturity date. All the random variables considered in this paper are $\BF_{T}$-measurable unless otherwise specified.

We will usually suppress $\omega\in\Omega$ for a random variable or a stochastic process. Moreover,
we use the following notations throughout the paper: \\[15pt]
\begin{tabular}{rl}
$Z'$ : & the transpose of any matrix or vector $Z$; \\
$\Vert Z\Vert$ : & $=\sqrt{Z'Z}$ for any column vector $Z$; \\
$\R^n$ : & the $n$-dimensional real Euclidean space; \\
$\mathbf{1}$ : & the unit (column) vector of $\R^{n}$ with all entries being 1;\\
$x^{-}$ : & $=\max\{-x,0\}$, the negative part of any number $x\in\R$; \\
$F_{X}(\cdot)$ : & the probability distribution function of any random variable $X$; \\
$F_{X}^{-1}(\cdot)$ : & the (right-continuous) inverse of $F_{X}(\cdot)$, also called the quantile\\
& (function) of $X$;\\
$\Phi(\cdot)$: & the distribution function of a standard normal random variable;\\
$X\sim \mu$ : & $X$ is a random variable following the distribution $\mu$;\\
$L^{p}(\BF_{t})$ : & the set of $\BF_{t}$-measurable random variables $X$ such that $\BE{|X|^p}<\infty$;\\
$L^{p}_{\BF}([0,T])$ : & the set of $\{{\cal F}_t\}_{t\geq0}$-adapted processes $\{X_t,\;t\geq 0\}$ such \\
& that $\BE{\int_0^T|X_t|^p\dt}<\infty$;\\
$\BE{f(\mu)}$ : &$=\int_{\R}f(x)\dd\mu(x)$ for a probability distribution function $\mu$ and \\
&function $f$, provided the integral is defined;\\
$D_{t}X$ : & the Malliavin derivative of any random variable $X$ for $t\in [0,T]$.
\end{tabular}\\[15pt]
We call a probability distribution function $\mu$ square-integrable if $\BE{\mu^2}<\infty$. Finally, sometimes we denote a stochastic process $X=\{X_t,\;t\geq0\}$ simply by $X$. 

In this paper, we call $g$ a {\it driver}
if $g: [0,T] \times \Omega\times \R\times \R^{n}\mapsto \R$ is $\{{\cal F}_t\}_{t\geq0}$-adapted and satisfies Lipschitz condition in $y$ and $z$, uniformly in $(t,\omega )$, and $g(t,0,0)\in L^{2}(\BF_{t})$ for $t\in[0,T]$.
Given a driver $g$ and a random variable $X\in L^{2}(\BF_{T})$, the following backward stochastic differential equation (BSDE)
\begin{align}\label{bsde1}
\begin{cases}
\dd Y_{t}=g(t,Y_{t},Z_{t})\dt+Z'_{t}\dd B_{t},\\
Y_{T}=X
\end{cases}
\end{align}
admits a unique square-integrable, adapted solution $(Y,Z)$ by the classical BSDE theory (see Pardoux and Peng \cite{PP90}). We call $Y_{0}$ the {\it $g$-expectation} of $X$, denoted by $\BE[g]{X}$. In particular, if $g$ is identical to zero, then $\BE[0]{X}=\BE{X}$ as in such a case $Y$ is a martingale.
\begin{remark}
One can also consider more general drivers for BSDEs and/or weaker conditions for the terminal value $X$, such as in space $L^{p}(\BF_{T})$ with $1\leq p<2$ (see \cite{BDHPS03}). However, if $X\in L^{1}(\BF_{T})$ for instance, then the existence and uniqueness of solutions to BSDE \eqref{bsde1} require more technical assumptions on the driver $g$. We are not pursuing that direction in this paper.
\end{remark}
\begin{defn}[Reachable distribution]
A distribution $\mu$ is called reachable (or replicable) under a driver $g$ if there exists $(X, Y,Z)$ satisfying \eqref{bsde1} with $X\sim\mu$. In this case, we say that $\mu$ is reached (or replicated) under the driver $g$.
\end{defn}
By the classical BSDE theory, every square-integrable distribution is reachable under any driver $g$.
Yet,
for a reachable distribution $\mu$, there are typically many processes reaching it, because there can be many different random variables $X$ having the same distribution $\mu$ in the BSDE \eqref{bsde1}. Naturally, then,
we are interested in the one with the smallest initial value $\BE[g]{X}$. This is related to the notion of {\it cost efficiency} which dates back to Cox and Leland \cite{CL82} and Dybvig \cite{D98a,D98b};
hence the following definition.
\begin{defn}[Cost efficiency and $g$-expectation of distribution]
Let $\mu$ be a reachable distribution and $g$ be a given driver.
A random variable $X^{*}$ is called (cost) efficient if it is an optimal solution to the following problem
\begin{align}\label{p1}
\inf_{X\sim \mu }\;\BE[g]{X}, 
\end{align}
and the infimum value is called the $g$-expectation of $\mu$, denoted by
\begin{align}\label{gmu}
\BE[g]{\mu}=\inf_{X\sim \mu }\;\BE[g]{X}. 
\end{align}
\end{defn}

A main purpose of this paper is to identify classes of drivers 
for which the efficient problem \eqref{p1} can be solved, along with its financial applications.
\par
As is well known, \eqref{p1} can be completely solved when the driver $g$ is linear. In turn,
the results on the linear case can help solve some nonlinear cases.

For ease of exposition, throughout this paper we assume all the integrals involved exist.
\begin{lem}
\label{linear}\label{lem1}
Assume that $g(t,y,z)\equiv r_{t}y+\theta _{t}'z+\delta _{t}$ in \eqref{bsde1}, where $r,\theta$ and $\delta$ are adapted processes.
Then
\[Y_{t}=\rho _{t}^{-1}\BE{\rho_{T}X-\int_{t}^{T}\delta _{s}\rho _{s}\ds\;\bigg|\;\BF_{t}},\]
where
\begin{align}\label{rho}
\rho_{t}=\exp \left(-\int_{0}^{t}\left(r_{s}+\tfrac{1}{2}\Vert\theta_{s}\Vert^{2}\right)\ds-\int_{0}^{t}\theta' _{s}\dd B_{s}\right)
\end{align}
is the unique solution of the SDE
\[\dd\rho_{t}=-r_{t}\rho _{t}\dt-\rho _{t}\theta'_{t}\dd B_{t},\qquad \rho_{0}=1.\]
In particular,
\[\BE[g]{X}=Y_0=\BE{\rho_{T}X}-\int_{0}^{T}\BE{\delta _{s}\rho _{s}}\ds.\]
\end{lem}
\begin{proof} Straightforward from applying It\^{o}'s formula to $\rho_{t}Y_{t}$.
\end{proof}

\begin{lem}[Hardy--Littlewood inequality]\label{HL}
Under the assumption of Lemma \ref{lem1}, if $\eta$ is an atomless random variable\footnote{A random variable is called atomless if its probability distribution is a continuous function.} and
$\mu$ is a distribution, then
\begin{align*}
\inf_{X\sim \mu }\;\BE{\eta X}=\BE{\eta X^{*}}=\int_{0}^{1}F_{\eta}^{-1}(1-p)\mu^{-1}(p)\ddp,
\end{align*}
where the efficient $X^{*}=\mu^{-1}(1-F_{\eta}(\eta))$ is the unique random variable with distribution $\mu$ that is anti-comonotonic with $\eta$.
\end{lem}
\par
The following result presents a complete answer to the problem \eqref{p1} when the driver is linear.
\begin{thm}\label{linear1}
\label{linear}
Under the assumption of Lemma \ref{lem1} and assuming further that
$\rho_{T}$ is atomless where $\rho_{t}$ is defined by \eqref{rho},
the $g$-expectation of $\mu$ is 
\begin{align*}
\BE[g]{\mu}\equiv \inf_{X\sim \mu }\;\BE[g]{X}
&=\BE[g]{X^*}=\int_{0}^{1}F_{\rho_{T}}^{-1}(1-p)\mu^{-1}(p)\ddp-\int_{0}^{T}\BE{\delta _{s}\rho _{s}}\ds,
\end{align*}
and $X^{*}=\mu ^{-1}(1-F_{\rho _{T}}(\rho _{T}))$ is the unique efficient terminal payoff.
\end{thm}
\begin{proof}
The results follow immediately from \citelem{lem1} and \citelem{HL}.
\end{proof}

\begin{remark}
The atomless assumption in the above results can be removed at the cost of efficient $X^{*}$ being not necessarily unique (see, e.g. \cite{X14}). Meanwhile,
if both $r$ and $\theta\neq 0$ are deterministic processes, then $\rho_{T}$ is log-normally distributed and thus automatically atomless.
\end{remark}


Now we proceed to study the efficient problem \eqref{p1} when the driver is nonlinear. While we are not yet able to
solve the nonlinear case in its greatest generality, we will discuss several special cases, each of which is either interesting theoretically in its own right or
has significant financial applications, in the next two sections.

The following comparison principle for BSDEs is well known (\cite{EPQ97}), which is needed for subsequent analysis.
\begin{lem}[Comparison principle]\label{compare}
Let $(Y^{i},Z^{i})$ be the solutions of the BSDE \eqref{bsde1} with parameters $(g,X)=(g_{i},X_{i})$, $i=1,2$, respectively.
If
\[X_{1}\geq X_{2},\quad g_{1}(t,Y_{t}^{1},Z_{t}^{1})\leq g_{2}(t,Y_{t}^{1},Z_{t}^{1}),
\quad\mbox{a.s. a.e.};\]
or
\[X_{1}\geq X_{2},\quad g_{1}(t,Y_{t}^{2},Z_{t}^{2})\leq g_{2}(t,Y_{t}^{2},Z_{t}^{2}),\quad\mbox{a.s. a.e.};\]
then $Y_{t}^{1}\geq Y_{t}^{2}$, a.s., $\forall t\in[0,T]$.
\end{lem}

\section{Two special cases}

In this section, we study two special cases.

\subsection{Case 1: $g(t,y,z)\equiv r_{t}y+\theta' _{t}z-(R_{t}-r_{t})(y-\mathbf{1}'(\sigma_{t}')^{-1}z)^{-}$} 
\noindent
We first consider a ``benchmark" case, which is important because its corresponding BSDE describes a wealth equation in a market with the loan rate $R_{t}$ generally higher than the deposit rate $r_{t}$: $R_{t}\geq r_{t},\;\;t\in[0,T]$.

Let us motivate the driver by considering a continuous-time market with multiple stocks and a bond. For $t\in[0,T]$, let $\sigma _{t}$ be the volatility matrix that is assumed to be non-singular, and $\theta _{t}$ be the risk premium of the market based on the deposit rate, i.e., $\theta _{t}=\sigma _{t}^{-1}(b_{t}-r_{t}\mathbf{1})$. Throughout this subsection we assume
\begin{assmp}\label{assmp1}
$R_{t}$, $r_{t}$, $\sigma _{t}$, $b_{t}$, and $\sigma _{t}^{-1}$ are deterministic, measurable functions of $t$ that are uniformly bounded.
\end{assmp}

\par
Let $\pi_{t}$ be the dollar amount invested in the risky assets at time $t$ in a small investor's portfolio.
Then it is known (\cite{EPQ97}) that her (self-financing) wealth process $Y=\{Y_{t},\;0\leq t\leq T\}$ to hedge a terminal random payoff $X$ follows a BSDE:
\begin{align}\label{diffBSDE}
\begin{cases}
\dd Y_{t}=\left[r_{t}Y_{t}+\theta' _{t}\sigma' _{t}\pi_{t}-(R_{t}-r_{t})(Y_{t}-\mathbf{1}'\pi_{t})^{-}\right]\dt+\pi'_{t}\sigma
_{t}\dd B_{t},\\
Y_{T}=X.
\end{cases}
\end{align}
We can regard $(Y_{t},Z_{t})\equiv (Y_{t},\sigma_{t}'\pi_{t})$ as the solution of the BSDE \eqref{bsde1} with the driver
\begin{align}\label{driver1}
g(t,y,z):= r_{t}y+\theta' _{t}z-(R_{t}-r_{t})(y-\mathbf{1}'(\sigma_{t}')^{-1}z)^{-}.
\end{align}

\noindent
In what follows we study the problem \eqref{p1} with the driver given by (\ref{driver1}),
and give some {\it sufficient} conditions on the terminal distribution $\mu$ under which \eqref{p1} can be solved and the $g$-expectation of $\mu$ expressed.
\par
Denote by $\mathbf{D}_{1,2}$ the set of random variables/vectors $\xi $ that
admit Malliavin derivatives $D_{t}\xi $ for a.e. $t\in \lbrack 0,T]$ with
\[\left\Vert \xi \right\Vert _{1,2}=\BE{\left\Vert \xi \right\Vert^{2}+\int_{0}^{T}\left\Vert D_{s}\xi \right\Vert ^{2}\ds} <+\infty.\]
\par

\begin{prop}\label{prop:bounds}
Suppose \citeassmp{assmp1} holds and $Y$ solves \eqref{diffBSDE} with $X\in \mathbf{D}_{1,2}$.
\begin{itemize}
\item
If $\mathbf{1}'(\sigma_{t}')^{-1}D_{t}X\geq X$, a.s. a.e., then $Y_{t}=
\overline{Y}_{t}$, a.s., $\forall t\in[0,T]$, where $(\overline{Y},\overline{Z})$ is the solution of the linear BSDE
\begin{align}\label{processoY}
\begin{cases}
\dd\overline{Y}_{t}=(r_{t}\overline{Y}_{t}+\theta' _{t}\overline{Z}_{t}+(R_{t}-r_{t})(\overline{Y}_{t}-\mathbf{1}'(\sigma'_{t})^{ -1}\overline{Z}_{t}))\dt+\overline{Z}'_{t}\dd B_{t},\\
\overline{Y}_{T}=X.
\end{cases}
\end{align}
Moreover, $\overline{Y}_{t}=\overline{\rho }_{t}^{-1}\BE{\overline{\rho }_{T}X|\mathcal{F}_{t}}$, a.s., $\forall t\in[0,T]$, where
\begin{multline}\label{orho}
\qquad\quad\overline{\rho}_{t}=\exp \bigg(-\int_{0}^{t}\left(R_{s}+\tfrac{1}{2}\Vert\theta_{s}-(R_{s}-r_{s})\sigma_{s}^{ -1}\mathbf{1}\Vert^{2}\right)\ds\\
-\int_{0}^{t}(\theta_{s}-(R_{s}-r_{s})\sigma_{s}^{ -1}\mathbf{1})'\dd B_{s}\bigg).\qquad\quad
\end{multline}

\item
If $\mathbf{1}'(\sigma_{t}')^{-1}D_{t}X\leq X$, a.s. a.e., then $Y_{t}=
\widetilde{Y}_{t}$, a.s., $\forall t\in[0,T]$, where $(\widetilde{Y},\widetilde{Z})$ is the solution of the linear BSDE%
\begin{align}\label{processwY}
\begin{cases}
\dd\widetilde{Y}_{t}=(r_{t}\widetilde{Y}_{t}+\theta' _{t}\widetilde{Z}_{t})\dt+\widetilde{Z}'_{t}\dd B_{t},\\
\widetilde{Y}_{T}=X.
\end{cases}
\end{align}
Moreover, $\widetilde{Y}_{t}= \rho_{t}^{-1}\BE{\rho_{T}X|\mathcal{F}_{t}}$, a.s., $\forall t\in[0,T]$, with $\rho_{t}$ defined by \eqref{rho}.
\end{itemize}
\end{prop}
\begin{proof}
We will prove the first case only, as the second one can be proved similarly.
Fix any $0\leq s<T$.
By Malliavin calculus, we have for $s\leq t\leq T$,
\begin{align}\label{processoDY}
\begin{cases}
\dd\; (D_{s}\overline{Y}_{t})=(r_{t}D_{s}\overline{Y}_{t}+\theta' _{t}D_{s}\overline{Z}_{t}\\
\qquad\qquad\quad+(R_{t}-r_{t})(D_{s}\overline{Y}_{t}-\mathbf{1}'(\sigma_{t}')^{-1}D_{s}\overline{Z}_{t}))\dt+(D_{s}\overline{Z}_{t})'\dd B_{t},\\
D_{s}\overline{Y}_{T} =D_{s}X.
\end{cases}
\end{align}
This BSDE is linear; so it is homogenous. Consequently,
\[(\breve{Y}_{t},\breve{Z}_{t})= (\mathbf{1}'(\sigma_{s}')^{-1}D_{s}\overline{Y}_{t}, \mathbf{1}'(\sigma_{s}')^{-1}D_{s} \overline{Z}_{t})= \mathbf{1}'(\sigma_{s}')^{-1}(D_{s}\overline{Y}_{t}, D_{s}\overline{Z}_{t}),\quad\mbox{a.s. a.e.}\] solves
\begin{align}\label{processbY}
\begin{cases}
\dd\breve{Y}_{t}=(r_{t}\breve{Y}_{t}+\theta' _{t}\breve{Z}_{t}+(R_{t}-r_{t})(\breve{Y}_{t}-\mathbf{1}'(\sigma'_{t})^{ -1}\breve{Z}_{t}))\dt+\breve{Z}'_{t}\dd B_{t},\\
\breve{Y}_{T}= \mathbf{1}'(\sigma_{s}')^{-1}D_{s}X.
\end{cases}
\end{align}
Because \[X\leq \mathbf{1}'(\sigma_{s}')^{-1}D_{s} X, \quad\mbox{a.s. a.e.},\] applying \citelem{compare} to \eqref{processoY} and \eqref{processbY}, we deduce $\overline{Y}_{t}\leq \breve{Y}_{t}$ a.s., $\forall t\in[s,T]$, In particular when $t=s$, it holds that
\[\overline{Y}_{s}\leq \breve{Y}_{s}=\mathbf{1}'(\sigma_{s}')^{-1}D_{s}\overline{Y}_{s}=\mathbf{1}'(\sigma_{s}')^{-1}\overline{Z}_{s},\]
where the last equation is due to $D_{s}\overline{Y}_{s}=\overline{Z}_{s}$. The above inequality holds for any $0\leq s<T$; so we may rewrite \eqref{processoY} as
\begin{align*}
\begin{cases}
\dd\overline{Y}_{t}=(r_{t}\overline{Y}_{t}+\theta' _{t}\overline{Z}_{t}-(R_{t}-r_{t})(\overline{Y}_{t}-\mathbf{1}'(\sigma'_{t})^{-1}\overline{Z}_{t})^{-})\dt+\overline{Z}'_{t}\dd B_{t},\\
\overline{Y}_{T}=X.
\end{cases}
\end{align*}
Comparing the above with the BSDE \eqref{diffBSDE} and using the uniqueness of solution, we conclude that $Y_{t}= \overline{Y}_{t}$.
Moveover, \citelem{lem1} gives the expression of $\overline{Y}_{t}$.
\end{proof}

\begin{prop}\label{bounds}
\label{uplow1}
Suppose \citeassmp{assmp1} holds.
Let the driver $g$ be defined by \eqref{driver1}. Then
\begin{align*}
\inf_{X\sim \mu }\;\BE[g]{X}\geq \max\left\{\BE{\rho_{T}\widetilde{X}}, \BE{\overline{\rho }_{T}\overline{X}}\right\},
\end{align*}%
where $\rho_{t}$ and $\overline{\rho }_{t}$ are defined by \eqref{rho} and \eqref{orho}, respectively,
$\widetilde{X}=\mu ^{-1}(1-F_{\rho _{T}}(\rho _{T}))$ and
$\overline{X}=\mu ^{-1}(1-F_{\overline{\rho} _{T}}(\overline{\rho} _{T}))$.
Moreover,
\begin{align*}
\BE[g]{\mu}\equiv \inf_{X\sim \mu }\;\BE[g]{X}=\BE[g]{\widetilde{X}}=\BE{\rho_{T}\widetilde{X}},
\end{align*}%
if $\mathbf{1}'\sigma _{t}^{-1}D_{t}\widetilde{X}\leq \widetilde{X}$, a.s. a.e.;
and
\begin{align*}
\BE[g]{\mu}\equiv \inf_{X\sim \mu }\;\BE[g]{X}=\BE[g]{\overline{X}}=\BE{\overline{\rho }_{T}\overline{X}},
\end{align*}
if $\mathbf{1}'\sigma _{t}^{-1}D_{t}\overline{X}\geq \overline{X}$, a.s. a.e..
\end{prop}
\begin{proof}
The second part of the proposition is an immediate consequence of the first part
along with \citeprop{prop:bounds} and \citethm{linear1}. So we only need to show the first part.
Let $g_{1}(t,y,z)= r_{t}y+\theta' _{t}z$ and
$g_{2}(t,y,z)\equiv r_{t}y+\theta' _{t}z+(R_{t}-r_{t})(y-\mathbf{1}'(\sigma_{t}')^{-1}z)$. Because $0\geq -a^-$ and $a=a^+-a^-\geq -a^-$, we have $g_1\geq g$ and $g_2\geq g$. So \citelem{compare} yields
\begin{equation*}
\BE[g]{X} \geq \BE[g_{1}]{X}, \quad
\BE[g]{X}\geq \BE[g_{2}]{X}
\end{equation*}
for any $X\sim\mu$. Notice that the drivers $g_{1}$ and $g_{2}$ are both linear; so \citethm{linear1} gives
\begin{equation*}
\inf_{X\sim \mu }\;\BE[g_{1}]{X}=\BE{\rho_{T}\widetilde{X}}, \quad
\inf_{X\sim \mu }\;\BE[g_{2}]{X}=\BE{\overline{\rho }_{T}\overline{X}}.
\end{equation*}
The desired result follows.
\end{proof}

\begin{thm} \label{optimal}
Suppose \citeassmp{assmp1} holds and the driver $g$ is defined by \eqref{driver1}.
\begin{itemize}
\item
Let $f(x)=\mu ^{-1}(1-F_{\rho_{T}}(x))$ and $c=\sup_{t\in[0,T]}\mathbf{1}'(\sigma_{t}')^{-1}\theta_{t}$. If
\begin{align}\label{suff2}
f(x)\geq -cxf_{x}(x),\quad\forall x>0,
\end{align}
then 
\begin{align}\label{suff0}
\BE[g]{\mu}=\BE{\rho_{T}\widetilde X}=\int_{0}^{1}F_{\rho_{T}}^{-1}(1-p)\mu^{-1}(p)\ddp,
\end{align}
where $\widetilde X=f(\rho_{T})$ and $\rho_{t}$ is defined by \eqref{rho}.
\item
Let $f(x)=\mu ^{-1}(1-F_{\overline{\rho}_{T}}(x))$ and $c=\inf_{t\in[0,T]}\mathbf{1}'(\sigma_{t}')^{-1}(\theta_{t}-(R_{t}-r_{t})\sigma_{t}^{-1}\mathbf{1})$.
If
\begin{align}\label{suff22}
f(x)\leq -c xf_{x}(x),\quad\forall x>0,
\end{align}
then 
\begin{align}\label{suff01}
\BE[g]{\mu}=\BE{\overline{\rho}_{T}\overline{X}}=\int_{0}^{1}F_{\overline{\rho}_{T}}^{-1}(1-p)\mu^{-1}(p)\ddp,
\end{align}
where $\overline{X}=\mu ^{-1}(1-F_{\overline{\rho}_{T}}(\overline{\rho}_{T}))$ and $\overline{\rho}_{t}$ is defined by \eqref{orho}.
\end{itemize}
\end{thm}

\begin{proof}
To prove the first claim \eqref{suff0}, by \citeprop{bounds},
it suffices to prove
\begin{align}\label{suff1}
\mathbf{1}'(\sigma_{t}')^{-1}D_{t} \widetilde X\leq \widetilde X, \quad\text{a.s. a.e.}.
\end{align}
Indeed, by the chain rule of Malliavin calculus and noting that $f_{x}\leq 0$, we obtain
\begin{align}\label{fineq}
\mathbf{1}'(\sigma_{t}')^{-1}D_{t} \widetilde X &=\mathbf{1}'(\sigma_{t}')^{-1}f_{x}(\rho_{T})D_{t}\rho_{T}\nn\\
&=f_{x}(\rho_{T})\mathbf{1}'(\sigma_{t}')^{-1}(-\rho_{T}\theta_{t})\nn\\
&=-\rho_{T}f_{x}(\rho_{T})\mathbf{1}'(\sigma_{t}')^{-1}\theta_{t}\nn\\
&\leq -c\rho_{T}f_{x}(\rho_{T})\nn\\
&\leq f(\rho_{T})=\widetilde X.
\end{align}
Hence \eqref{suff0} holds.

We now prove the second claim \eqref{suff01}. By \citeprop{bounds}, it is sufficient to prove
\begin{align}\label{suff1}
\mathbf{1}'(\sigma_{t}')^{-1}D_{t} \overline{X}\geq \overline{X}, \quad\text{a.s. a.e.}.
\end{align}
We rewrite $\overline{X}=f(\overline{\rho}_{T})$ with $f(x)=\mu ^{-1}(1-F_{\overline{\rho}_{T}}(x))$.
By the chain rule of Malliavin calculus,
\begin{align*}
\mathbf{1}'(\sigma_{t}')^{-1} D_{t} \overline{X} &=\mathbf{1}'(\sigma_{t}')^{-1}f_{x}(\overline{\rho}_{T})D_{t}\overline{\rho}_{T}\\
&=-\overline{\rho}_{T}f_{x}(\overline{\rho}_{T})\mathbf{1}'(\sigma_{t}')^{-1}(\theta_{t}-(R_{t}-r_{t})\sigma_{t}^{-1}\mathbf{1})\\
&\geq-c\overline{\rho}_{T}f_{x}(\overline{\rho}_{T})\\
&\geq f(\overline{\rho}_{T})=\overline{X},
\end{align*}
thanks to $f_{x}\leq 0$ and $\mathbf{1}'(\sigma_{t}')^{-1}(\theta_{t}-(R_{t}-r_{t})\sigma_{t}^{-1}\mathbf{1})\geq c$. The proof is complete.
\end{proof}

Let us now apply the above results to an expected utility maximization portfolio selection problem in the aforementioned market with different deposit and loan rates.
\par
We aim to find the best portfolio process $\pi$ to
\begin{align*} 
\max_{\pi} &\quad \BE{u(Y_{T})},\\
\mathrm{subject\; to}&\quad Y_{0}=x,\quad Y_{T}\geq 0,\nn
\end{align*}
where the wealth process $Y$ follows \eqref{diffBSDE}, and $u$ is a given increasing and concave utility function defined on $[0,\infty)$.
\par
This problem can be solve by the following two steps: first finding a best terminal payoff $X^{*}$ that solves
\begin{align} \label{eut0}
\max_{X} &\quad \BE{u(X)},\\
\mathrm{subject\; to}&\quad \BE[g]{X}=x,\quad X\geq 0,\nn
\end{align}
and then determining a portfolio $\pi$ to replicate $X^{*}$. Since the latter problem
is straightforward by the BSDE theory, we focus on the former.
\begin{thm}
If $\mathbf{1}'(\sigma_{t}')^{-1}\theta_{t}\leq 0$ for all $0\leq t\leq T$, then an optimal payoff of the problem \eqref{eut0} is given by
\[X^{*}=\left(u'\right)^{-1}(\lambda \rho_{T}),\]
where the Lagrange multiplier $\lambda>0$ is uniquely determined by $\BE{\rho_{T}X^{*}}=x$ and $\rho_t$ is defined by \eqref{rho}.
\end{thm}
\begin{proof}
Because the objective of the problem \eqref{eut0} is law-invariant and increasing in $X$, an optimal solution $X^{*}$ can be found among those that are cost efficient, namely those that solve the following problem
\begin{align*}
\inf_{X\sim \mu^{*} }\;\BE[g]{X},
\end{align*}
where $\mu^{*}$ is the distribution of $X^{*}$. In view of the constraint $X\geq 0$ (so that $\mu^{*}(0-)=0$) and the assumption $\mathbf{1}'(\sigma_{t}')^{-1}\theta_{t}\leq 0$ for all $0\leq t\leq T$ (so that \eqref{suff2} holds), we deduce from \citethm{optimal} that $X^*$ can be represented in terms of its quantile:
\[X^{*}=(\mu^{*})^{-1}(1-F_{\rho_{T}}(\rho_{T})).\]
The rest is to determine $(\mu^{*})^{-1}$, which is an optimal solution of the following so-called quantile optimization problem
\begin{align} \label{qt1}
\max_{\mu^{-1}} &\quad \int_{R}u(x)\dd \mu(x)\equiv \int_{0}^{1}u(\mu^{-1}(p))\ddp,\\
\mathrm{subject\; to}&\quad \int_{0}^{1}F_{\rho_{T}}^{-1}(1-p)\mu^{-1}(p)\ddp=x,\quad \mu(0-)=0,\nn
\end{align}
where the constraint follows again from \citethm{optimal} that
\[\int_{0}^{1}F_{\rho_{T}}^{-1}(1-p)(\mu^{*})^{-1}(p)\ddp=\BE[g]{X^{*}}=x.\]
With $\mu^{-1}$ as the decision variable, this problem can be solved easily by the Lagrange method. The solution is
\[(\mu^{*})^{-1}(p)=(u')^{-1}(\lambda F_{\rho_{T}}^{-1}(1-p)),\]
for some constant Lagrange multiplier $\lambda>0$ satisfying
\[\int_{0}^{1}F_{\rho_{T}}^{-1}(1-p)(u')^{-1}(\lambda F_{\rho_{T}}^{-1}(1-p))\ddp=x.\]
Consequently, \[X^{*}=(\mu^{*})^{-1}(1-F_{\rho_{T}}(\rho_{T}))=(u')^{-1}(\lambda\rho_{T}).\]
Since
\[\BE{\rho_{T}X^{*}}=\BE{\rho_{T}(u')^{-1}(\lambda\rho_{T})}
=\int_{0}^{1}F_{\rho_{T}}^{-1}(p)(u')^{-1}(\lambda F_{\rho_{T}}^{-1}(p))\ddp=x,\]
the proof is complete.
\end{proof}

Note the result is under the assumption that $\mathbf{1}'(\sigma_{t}')^{-1}\theta_{t}\leq 0$. As an example of this assumption, take
\[\sigma_{t}=
\begin{pmatrix}
4 &\quad 3\\
3 &\quad 4
\end{pmatrix},\quad\theta_{t}=
\begin{pmatrix}
1\\
-2
\end{pmatrix}.
\]
One can verify that the assumption holds.

\subsection{Case 2: $g(t,y,z)=- A'_{t}z^{+}- C'_{t}z^{-}\leq0$}
\noindent
In this section, we study the case when drivers satisfy the following
\begin{assmp}\label{assmp2}
The driver $ g$ is a deterministic function of $(t,z)$, and satisfies
\begin{itemize}
\item Non-positivity: $g(t,z)\leq 0$;
\item Lipschitz continuity: $|g(t,z_{1})-g(t,z_{2})|\leq K\Vert z_{1}-z_{2}\Vert$, for some $K>0$;
\item Positive homogeneity in $z$: $g(t,\alpha z)=\alpha g(t,z)$, for any $\alpha\geq 0$.
\end{itemize}
\end{assmp}
For any $ g(t,z)$ satisfying \citeassmp{assmp2}, we have $|g(t,z)|\leq K\Vert z\Vert$ by its non-positivity and Lipschitz continuity.
\begin{remark}
A driver $g$ that satisfies \citeassmp{assmp2} must be of the form
\[g(t,z)= - A'_{t}z^{+}- C'_{t}z^{-},\quad\mbox{a.s. a.e.},\]
where any $k$-th component of $z^{+}$ (resp. $z^{-}$) is the positive (resp. negative) part of the $k$-th component of $z$, $A_{t}$ and $C_{t}$ are vectors with nonnegative components only and $\Vert A_{t}\Vert \leq K $ and $\Vert C_{t} \Vert \leq K$. In particular, it covers the special case when $g(t,z)= - A'_{t}|z|$
where $|z|$ is the vector whose components are the absolute values of the corresponding components of $z$.
\end{remark}

\begin{remark}
In this case, the driver is assumed to be independent of $y$. This assumption is valid in some financial applications. For example, with a single deterministic interest rate, the discounted wealth process satisfies a BSDE whose driver does not depend on $y$.
\end{remark}

\begin{thm}\label{sublinear}
If the driver $g$ satisfies \citeassmp{assmp2}, then 
\begin{align*}
\BE[g]{\mu}\equiv \inf_{X\sim \mu }\;\BE[g]{X}=\BE{\mu}
\end{align*}
for any distribution $\mu$.
\end{thm}
\begin{proof}
By \eqref{bsde1}, we have that
\[\BE[g]{X}=X-\int_{0}^{T}g(t,Z_{t})\dt-\int_{0}^{T}Z'_{t}\dd B_{t}.\]
Hence
\[\BE[g]{X}=\BE{X}-\int_{0}^{T}\BE{g(t,Z_{t})}\dt\geq \BE{X}=\BE{\mu}\]
for any $X\sim\mu$, as $g$ is non-positive. Therefore,
\begin{align*}
\inf_{X\sim \mu }\;\BE[g]{X}\geq \BE{\mu}.
\end{align*}
\par
To show the reverse inequality, we set
\[X^{\alpha}=\mu^{-1}(\Phi(\alpha^{-1/2}B_{\alpha})),\quad \text{ for }0<\alpha\leq T.\]
Then it is easily seen that $X^{\alpha}$ is $\BF_{\alpha}$-measurable following the distribution $\mu$.
If we can show
\begin{align}\label{ineq00}
\lim_{\alpha\to 0+}\BE[g]{X^{\alpha}}=\BE{\mu},
\end{align}
then the desired result follows.
Let $(Y^{\alpha}, Z^{\alpha})$ solve \eqref{bsde1} with the terminal condition $Y^{\alpha}_{T}=X^{\alpha}$. Then
\[\BE[g]{X^{\alpha}}=X^{\alpha}-\int_{0}^{T}g(t,Z^{\alpha}_{t})\dt-\int_{0}^{T}\left(Z^{\alpha}_{t}\right)'\dd B_{t}.\]
By non-linear Feymann-Kac formula (\cite{P92}), we have $\BE[g]{X^{\alpha}}=Y^{\alpha}_{0}=u^{\alpha}(0,0)$, where $u^{\alpha}(t,x)$ is the solution of
\begin{align}\label{pde1}
\begin{cases}
\frac{\partial u^{\alpha}}{\partial t}+\triangle u^{\alpha}+g(t,\nabla u^{\alpha})=0,\\
u^{\alpha}(\alpha,x)=\mu^{-1}(\Phi (\alpha^{-1/2} x)).
\end{cases}
\end{align}
Set $\overline{u}^{\alpha}(t,x)=u^{\alpha}(\alpha t,\alpha^{1/2}x)$, then
\begin{align}\label{pde2}
\begin{cases}
\frac{\partial\overline{u}^{\alpha}}{\partial t}+\triangle\overline{u}^{\alpha}+\alpha^{1/2} g(t,\nabla\overline{u}^{\alpha})=0,\\
\overline{u}^{\alpha}(1,x)=\mu^{-1}(\Phi (x)),
\end{cases}
\end{align}
because $g(t, z)$ is positively homogeneous in $z$.
\par
Denote $\overline{Y}_{t}^{\alpha}=\overline{u}^{\alpha}(t,B_{t})$, which satisfies the BSDE
\[\overline{Y}_{t}^{\alpha} =\overline{Y}_{1}^{\alpha}-\alpha^{1/2}\int_{t}^{1}g(s,\overline{Z}_{s}^{\alpha})\dt-\int_{t}^{1}(\overline{Z}_{s}^{\alpha})'\dd B_{s},
\quad t\in[0,1],\]
for some adapted process $\overline{Z}^{\alpha}$. Consequently,
\[\BE{\overline{Y}_{0}^{\alpha}} =\BE{\overline{Y}_{1}^{\alpha}}-\alpha^{1/2}\int_{0}^{1}\BE{g(s,\overline{Z}_{s}^{\alpha})}\ds.\]
Noting $\overline{Y}_{0}^{\alpha}=\overline{u}^{\alpha}(0,0)=u^{\alpha}(0,0)=Y^{\alpha}_{0}=\BE[g]{X^{\alpha}}$
and $\overline{Y}_{1}^{\alpha}=\overline{u}^{\alpha}(1,B_{1})=\mu^{-1}(\Phi (B_{1}))\sim \mu$, the above equation can be rewritten as
\begin{align} \label{ineq0}
\BE[g]{X^{\alpha}}=\BE{\mu}-\alpha^{1/2}\int_{0}^{1}\BE{g(t,\overline{Z}_{t}^{\alpha})}\dt.
\end{align}
If we can prove that there exists a constant $C$ independent of $\alpha$ ($\leq T$) such that
\begin{align} \label{ineq1}
\int_{0}^{1}\BE{\Vert \overline{Z}_{t}^{\alpha}\Vert^{2}}\dt\leq C,
\end{align}
then applying Cauchy's inequality and using $|g(t,z)|\leq K\Vert z\Vert$ we obtain
\begin{align*}
\bigg|\int_{0}^{1}\BE{g(t,\overline{Z}_{t}^{\alpha})}\dt\bigg|\leq \sqrt{\int_{0}^{1}\BE{(g(t,Z^{\alpha}_{t}))^{2}}\dt}
\leq \sqrt{\int_{0}^{1} K^{2}\BE{\Vert \overline{Z}_{t}^{\alpha}\Vert^{2}}\dt}\leq K\sqrt{C}.
\end{align*}
By virtue of this inequality, sending $\alpha\to 0$ in \eqref{ineq0} then leads to the desired result \eqref{ineq00}.
\par
So it remains to prove \eqref{ineq1}.
Applying It\^{o}'s lemma to $(\overline{Y}^{\alpha}_{t})^{2}$, we have
\[
\dd\;(\overline{Y}^{\alpha}_{t})^{2}=\left(2\alpha^{1/2}\overline{Y}^{\alpha}_{t}g(t,\overline{Z}^{\alpha}_{t})+\Vert \overline{Z}^{\alpha}_{t}\Vert^{2}\right)\dt+2\overline{Y}^{\alpha}_{t}\left(\overline{Z}^{\alpha}_{t}\right)'\dd B_{t},
\]
or
\[
(\overline{Y}^{\alpha}_{1})^{2}=(\overline{Y}^{\alpha}_{t})^{2}+\int_{t}^{1}\left(2\alpha^{1/2}\overline{Y}^{\alpha}_{s}g(s,\overline{Z}^{\alpha}_{s})+\Vert \overline{Z}^{\alpha}_{s}\Vert^{2}\right)\ds+\int_{t}^{1} 2\overline{Y}^{\alpha}_{s}\left(\overline{Z}^{\alpha}_{s}\right)'\dd B_{s}.
\]
Taking expectation on both sides and rearranging,
\begin{align*}
\BE{(\overline{Y}^{\alpha}_{t})^{2}}+\int_{t}^{1}\BE{\Vert \overline{Z}^{\alpha}_{s}\Vert^{2}}\ds &=\BE{(\overline{Y}^{\alpha}_{1})^{2}}-\int_{t}^{1}2\alpha^{1/2}\BE{\overline{Y}^{\alpha}_{s}g(s,\overline{Z}^{\alpha}_{s})}\ds \\
&\leq \BE{\mu^{2}}+2\alpha K^{2}\int_{t}^{1}\BE{(\overline{Y}^{\alpha}_{s})^{2}}\ds+\frac{1}{2K^{2}}\int_{t}^{1}\BE{(g(s,\overline{Z}^{\alpha}_{s}))^{2}}\ds\\
&\leq \BE{\mu^{2}}+2T K^{2}\int_{t}^{1}\BE{(\overline{Y}^{\alpha}_{s})^{2}}\ds+\frac{1}{2}\int_{t}^{1}\BE{\Vert \overline{Z}^{\alpha}_{s}\Vert^{2}}\ds,
\end{align*}
leading to
\begin{align} \label{ineq3}
\BE{(\overline{Y}^{\alpha}_{t})^{2}}+\frac{1}{2}\int_{t}^{1}\BE{\Vert \overline{Z}^{\alpha}_{s}\Vert^{2}}\ds
&\leq \BE{\mu^{2}}+2T K^{2}\int_{t}^{1}\BE{(\overline{Y}^{\alpha}_{s})^{2}}\ds.
\end{align}
Therefore, \begin{align*}
\BE{(\overline{Y}^{\alpha}_{t})^{2}} &\leq \BE{\mu^{2}}+2T K^{2}\int_{t}^{1}\BE{(\overline{Y}^{\alpha}_{s})^{2}}\ds.
\end{align*}
Gronwall's inequality yields
\begin{align*}
\max_{0\leq t\leq 1}\BE{(\overline{Y}^{\alpha}_{t})^{2}} &\leq C_1
\end{align*}
for some constant $C_1$ independent of $\alpha$. This together with \eqref{ineq3} immediately gives the desired inequality \eqref{ineq1} with some constant $C$ independent of $\alpha$.
\end{proof}

\begin{remark}\label{examples}
It follows from \eqref{ineq0} that $\{X^{\alpha}\}$ is a minimizing sequence of the problem \eqref{p1} as $\alpha\rightarrow0$. 
However, the infimum of \eqref{p1} may or may not be attained. Here let us give examples for both situations. As the first example, select $A$ and $C$ 
such that the set $\sets=\{(\omega,t): A_{t}=0\}\cup \{(\omega,t): C_{t}=0\}$ has a positive $\dd\bp\times \dt$ measure. Fix $Z\in L^{2}_{\BF}([0,T])$ that is supported on $\sets$ along with a real constant $a$, and define $ X^*=a+\int_0^Tg(t,Z_{t})\dt+\int_0^TZ'_{t}\dd B_{t}$. Let $\mu$ be the distribution of $X^*$.
Then
\[g(t,Z_{t})= - A'_{t}Z_{t}^{+}- C'_{t}Z_{t}^{-}=0, \quad\mbox{a.s. a.e.},\]
so
\[\BE[g]{X^*}=\BE{X^*}-\int_{0}^{T}\BE{g(t,Z_{t})}\dt=\BE{X^*}=\BE{\mu},\]
which means that $X^*$ is efficient for $\mu$ under $g$ or that the infimum of \eqref{p1} is attained. For the other example, take $A$ and $C$
such that the same set $\sets$ has zero $\dd\bp\times \dt$ measure (namely, $A_{t}$ and $C_{t}$ have only positive components). Let $\mu$ be any measure other than Dirac.
Then for any random variable $X\sim\mu$, we must have that $Z$ is not identical to 0 in \eqref{bsde1}. Hence,
\[\int_{0}^{T}\BE{g(t,Z_{t})}\dt=\int_{0}^{T}\BE{- A'_{t}Z_{t}^{+}- C'_{t}Z_{t}^{-}}\dt<0,\]
yielding
\[\BE[g]{X}=\BE{X}-\int_{0}^{T}\BE{g(t,Z_{t})}\dt>\BE{X}=\BE{\mu}.\]
So the infimum of \eqref{p1} is not achieved and there is no efficient solution.
\end{remark}

\noindent
We now apply the results to a portfolio selection problem. Consider a financial market with one risk-free asset and one risky asset (such as a stock index).
Without loss of generality we assume the risk-free rate $r_t\equiv 0$ (otherwise we consider discounted values in our analysis below) and the volatility rate of the risky asset $\sigma_t\equiv 1$. Assume it is in a bear market now so the expected excess rate of return $\theta_t\leq 0$ for all $t\in[0,T]$. Taking a short position in the risky asset may be tempting, but it incurs a non-negligible transaction fee.
The wealth process in this market replicating a terminal position $X$ is
\begin{align}\label{bsde2}
\begin{cases}
\dd Y_{t}=\left(\theta_{t}\pi_{t}-K_t\pi^{-}_{t}\right)\dt+\pi_{t}\dd B_{t},\\
Y_{T}=X,
\end{cases}
\end{align}
where $K_t$ is the unit transaction cost of shorting at $t$ and we assume that it is large enough so that $K_t\geq |\theta_t|$.
Then $(Y,Z)\equiv (Y,\pi)$ is the solution of the BSDE \eqref{bsde2} with the driver
\begin{align}\label{driver10}
g(t,z)=\theta_{t}z^{+}- (K_{t}+\theta_{t})z^{-}.
\end{align}
Because $\theta_t\leq 0$ and $K_t\geq |\theta_t|$, this driver satisfies \citeassmp{assmp2}.

\par
Consider the following dynamic portfolio choice problem under rank-dependent utility (RDU):
\begin{align}\label{objective02}
\sup_{\pi(\cdot)} \quad & \int_0^{\infty} u(x)\dd\:\big(1-w(1-F_{Y_{T}}(x))\big),\\
\nonumber\textrm{subject to}\quad & (Y, \pi) \text{ follows dynamic of \eqref{bsde2} with initial wealth $Y_0$ },\quad Y_{T}\geqslant 0.
\end{align}
Introduce the following auxiliary static optimization problem in terms of $X$
\begin{align}\label{objective12}
\sup_{X} \quad & \int_0^{\infty} u(x)\dd\:\big(1-w(1-F_{X}(x))\big),\\
\nonumber\textrm{subject to}\quad & \BE{X}=Y_{0},\quad X\geqslant 0.
\end{align}
\par
Taking the quantile $\mu^{-1}$ of $X$ as the new decision variable, 
we can rewrite \eqref{objective12}, via a simple calculation, as
\begin{align}
\sup_{\mu} \quad & \int_0^{1} u(\mu^{-1}(p))w'(1-p)\ddp,\\
\nonumber\textrm{subject to}\quad & \int_0^{1} \mu^{-1}(p) \ddp=Y_{0}, \quad \mu(0-)=0.
\end{align}
This is a quantile optimization problem, whose solution is given by Xia and Zhou \cite{XZ16} and Xu \cite{X16}:

\[(\mu^{*})^{-1}(p)=(u')^{-1}\big(\lambda\phi'(1-p)\big)\]
where $\phi$ is the concave envelope of the function $1-w^{-1}(1-p)$ on $[0,1]$ and $\lambda>0$ is determined by the budget constraint
\[\int_0^{1} (u')^{-1}\big(\lambda\phi'(1-p)\big)\ddp=Y_{0}.\]

By \citeremark{examples}, if
\[ X^*:=Y_0+\int_0^TZ'_{t}\dd B_{t}\sim \mu^{*}\]
for some $Z$ supported on the set
\[\sets=\{(\omega,t): \theta_{t}=0\}\cup \{(\omega,t): K_{t}+\theta_t=0\},
\]
then
\[ \BE[g]{X^*}=\inf_{X\sim \mu }\;\BE[g]{X}=\BE{\mu^*}=\BE{X^*}= \BE[g]{\mu^*} \]
and $X^*$ is efficient. Consequently, optimal portfolio $\pi^{*}$ of the problem \eqref{objective02} can be determined by the solution of the BSDE \eqref{bsde2} with $X=X^{*}$.


\section{Law-invariant $g$-expectation}
\noindent
As mentioned early, given a reachable distribution, there may be many processes reaching it under the same driver.
But the costs of these processes to reach the same distribution, in general, are different; hence the cost efficient problem. However, there exist drivers under which the costs to replicating a given distribution $\mu$ are invariant for all replicating processes. In this case one does not even need to solve the efficient problem, and the $g$-expectation of that distribution is the initial cost of (any) replicating process. A trivial such example is when the driver is $g\equiv 0$, under which $\BE[g]{X}=\BE[0]{X}=\BE{X}=\BE{\mu}$ for any $X\sim\mu$.
This motivates us to propose the following definition.
\begin{defn} [Law-invariant $g$-expectation]
Given a distribution $\mu$, a driver $g$ is called $\mu$-invariant, if $\BE[g]{X}$ are the same for all $X\sim \mu$. 
A $g$-expectation (or its driver $g$) is called law-invariant if it is $\mu$-invariant for any distribution $\mu$.
\end{defn}

\begin{remark}
If a driver $g$ is $\mu$-invariant, then $\BE[g]{\mu}=\BE[g]{X}$ for any $X\sim\mu$.
\end{remark}

\begin{remark}
A driver $g$ is $\mu$-invariant if and only if every admissible solution of \eqref{p1} for $(g,\mu)$ is also an optimal solution.
\end{remark}
\begin{remark}
Every deterministic driver $g$ is $\mu$-invariant if $\mu$ is a Dirac measure, because in such a case there is only one (deterministic) replicating process by virtue of the uniqueness of the solution to the corresponding BSDE.
\end{remark}

In this section we investigate under what conditions a $g$-expectation is law-invariant. Note that throughout the section we will relax the restriction of drivers being Lipschitz continuous by allowing them to have quadratic growth in $z$.

\subsection{A simple case}

\par
The following is a simple but non-trivial example of law-invariant $g$-expectation, which will in turn inspire the subsequent Proposition \ref{lawinvariant1}.

\begin{eg}
The driver $g(t,y,z)\equiv -\frac{1}{2}z^{2}$ is law-invariant. In fact, let $Y$ satisfy the BSDE \eqref{bsde1} with this driver $g$. Then It\^{o}'s formula yields
\begin{align*}
\dd \; e^{Y_{t}}=e^{Y_{t}}Z'_{t}\dd B_{t}.
\end{align*}
Hence $e^{Y_{0}}=\BE{ e^{Y_{T}}}=\BE{ e^{\mu}}$ if $Y_{T}\sim\mu$, implying $\BE[g]{X}=\log \big(\BE{ e^{\mu}}\big)$ for any $X\sim\mu$, which is $\mu$-invariant.
\end{eg}
\begin{prop}\label{lawinvariant1}
Any time-invariant driver of the form $g(t,y,z)\equiv -f(y)\Vert z\Vert^{2}$, 
is law-invariant and
\begin{align*}
\BE[g]{\mu}=\varphi^{-1}\Big(\BE{\varphi(\mu)}\Big),
\end{align*}
for any distribution $\mu$ 
where
\[\varphi(x)=\int_{0}^{x}\exp\bigg(2\int_{0}^{y}f(s)\ds\bigg)\dy.\]
\end{prop}
\begin{proof}
Let $Y$ follow the BSDE \eqref{bsde1} with $g(t,y,z)\equiv -f(y)\Vert z\Vert^{2}$. Note that $\varphi$ satisfies the following ordinarily differential equation (ODE):
\[2f(x)\varphi_{x}(x)-\varphi_{xx}(x)=0.\]
Hence, It\^{o}'s formula yields
\begin{align*}
\dd \;\varphi(Y_{t})= \varphi_{x}(Y_{t})Z'_{t}\dd B_{t}.
\end{align*}
This leads to $\varphi(Y_{0})=\BE{\varphi(Y_{T})}=\BE{\varphi(\mu)}$ if $Y_{T}\sim\mu$, implying that the driver $g$ is $\mu$-invariant because $\varphi$ is strictly increasing.
\end{proof}
\begin{remark}
One may wonder if any time-invariant and law-invariant driver must be of the form in Proposition \ref{lawinvariant1}. Unfortunately this is not true.
For instance, take $g(t,y,z)\equiv ry$ where $r$ is a fixed constant.
In this case $e^{-rt}Y_{t}$ is a martingale; so $\BE[g]{X}=e^{-rT}\BE{X}=e^{-rT}\BE{\mu}$ for any $X\sim\mu$ suggesting that $g$ is law-invariant.
In the same spirit, one can also show that $g(t,y,z)\equiv r_{t}y$ is law-invariant, for any determinist function $r_{t}$. This shows that time-variant drivers can also be law-invariant.
\end{remark}
\par
In the next subsection, we will give a large class of time-variant and law-invariant drivers.

\subsection{A class of law-invariant drivers}
\noindent
In this section, we look for sufficient conditions to ensure a driver to be law-invariant. Identifying necessary conditions is a more challenging and open problem.
\par
Specifically, we study under what conditions deterministic divers of the form (called $z$-separable drivers) \[g(t,y,z)\equiv f(t,y)k(z)+h(t,y)\] are law-invariant. \par
Our idea is as follows. If we find a function $\varphi$ such that $\varphi(t,Y_{t})$ is a martingale, then
\[\varphi(0,Y_{0})=\BE{\varphi(T,Y_{T})}=\BE{\varphi(T,\mu)}\] for any $Y_{T}\sim\mu$.
If, in addition, the function $y\mapsto \varphi(0,y)$ is strictly increasing, then $Y_{0}$ is uniquely determined by the above equation and, hence, the driver $g$ is law-invariant. So our target becomes to find conditions on $ g$ that guarantees the existence of such a function $\varphi$.
\par
In what follows, we first intuitively derive some conditions, and then prove that those conditions are indeed sufficient.
\par
Suppose $(Y,Z)$ follows the BSDE \eqref{bsde1}. Then It\^{o}'s lemma yields
\begin{align}\label{itoequation}
\dd \varphi(t,Y_{t})= \left[\varphi_{t}(t,Y_{t})+\varphi_{y}(t,Y_{t})g(t,Y_{t},Z_{t})+\tfrac{1}{2}\varphi_{yy}(t,Y_{t})\Vert Z_{t}\Vert^{2}\right]\dt+\varphi_{y}(t,Y_{t})Z'_{t}\dd B_{t}.
\end{align}
Hence,  $\varphi(t,Y_{t})$ is a martingale if 
\[\varphi_{t}(t,y)+\varphi_{y}(t,y)h(t,y)\equiv0,\quad \varphi_{y}(t,y)f(t,y) k(z)+\tfrac{1}{2}\varphi_{yy}(t,y)\Vert z\Vert^{2}\equiv 0,\]
along with some integrability condition on $\varphi_{y}(t,Y_{t})Z'_{t}$. Ignore the integrability condition for now in our intuitive argument here and focus on the two equations above. 
Intuitively, the second equation suggests $k(z)\equiv \Vert z\Vert^{2}$ (up to a multiplier which can be absorbed by $f$). Hence
\begin{align}\label{eqlt0}
\varphi_{t}(t,y)+\varphi_{y}(t,y)h(t,y)\equiv0,\quad \varphi_{y}(t,y)f(t,y) +\tfrac{1}{2}\varphi_{yy}(t,y)\equiv 0,
\end{align}
which imply
\begin{align}
\varphi_{ty}(t,y)&=-\varphi_{yy}(t,y)h(t,y)-\varphi_{y}(t,y)h_{y}(t,y) =\varphi_{y}(t,y)(2f(t,y)h(t,y)-h_{y}(t,y)),\label{eqlt1}\\
\label{eqlt2} \varphi_{y}(t,y)&= \varphi_{y}(t,0)e^{-2\int_{0}^{y}f(t,x)\dx}.
\end{align}
Differentiating (\ref{eqlt2}) in $t$ gives
\begin{align}\label{eqlt3}
\varphi_{ty}(t,y)=\left (\varphi_{ty}(t,0)-2\varphi_{y}(t,0)\int_{0}^{y}f_{t}(t,x)\dx\right)e^{-2\int_{0}^{y}f(t,x)\dx}.
\end{align}
By comparing \eqref{eqlt1} and \eqref{eqlt3}, we get
\begin{multline}\label{eqlt4}
\qquad\varphi_{y}(t,y)(2f(t,y)h(t,y)-h_{y}(t,y))\\
=\left (\varphi_{ty}(t,0)-2\varphi_{y}(t,0)\int_{0}^{y}f_{t}(t,x)\dx\right)e^{-2\int_{0}^{y}f(t,x)\dx}.\qquad
\end{multline}
When $y=0$, it becomes
\begin{align}\label{eqlt7}
\varphi_{y}(t,0)(2f(t,0)h(t,0)-h_{y}(t,0))= \varphi_{ty}(t,0).
\end{align}
Substituting \eqref{eqlt7} into \eqref{eqlt4} gives
\begin{multline}\label{eqlt8}
\varphi_{y}(t,y)(2f(t,y)h(t,y)-h_{y}(t,y)) \\
= \varphi_{y}(t,0)\left(2f(t,0)h(t,0)-h_{y}(t,0)-2\int_{0}^{y}f_{t}(t,x)\dx\right)e^{-2\int_{0}^{y}f(t,x)\dx}.
\end{multline}
Using \eqref{eqlt2}, we deduce from the above that
\begin{align}\label{eqlt6}
2f(t,y)h(t,y)-h_{y}(t,y)= 2f(t,0)h(t,0)-h_{y}(t,0)-2\int_{0}^{y}f_{t}(t,x)\dx.
\end{align}
Taking derivative in $y$ yields
\begin{align}\label{eqlt5}
2f_{y}(t,y)h(t,y)+ 2f(t,y)h_{y}(t,y)-h_{yy}(t,y)= -2f_{t}(t,y).
\end{align}

\begin{remark}
One can see that $f$ must be time-invariant if $h\equiv 0$ in \eqref{eqlt5}. This case has been studied in Proposition \ref{lawinvariant1}. In general, we can always find a time-invariant $f$ satisfying \eqref{eqlt5} if $h$ is time-invariant.
\end{remark}

The following theorem formalizes the above intuitive PDE argument and stipulates that \eqref{eqlt5} is sufficient for $g(t,y,z)\equiv f(t,y) \Vert z\Vert^{2}+h(t,y)$ to be law-invariant.

\begin{thm}\label{THMlawin}
Suppose that  $f\in C^{1,1}([0,T]\times\R)$ and $h\in C^{0,2}([0,T]\times\R)$ satisfy \eqref{eqlt5}
and
\[\sup\{|f(t,y)|:(t,y)\in [0,T]\times\R\}< \alpha,\quad \sup\{|h(t,y)|:(t,y)\in [0,T]\times\R\}\leq \beta.\]
 Then the driver $g(t,y,z)\equiv f(t,y)\Vert z\Vert^{2}+h(t,y)$ is law-invariant among those distributions $\mu$ satisfying
\[\BE{e^{2\alpha e^{\beta T}|\mu|}}<\infty,\]
in which case the $g$-expectation $\BE[g]{\mu}$ is uniquely determined by \[\varphi(0,\BE[g]{\mu})=\BE{\varphi(T,\mu)}\] where
\begin{multline}
\qquad\qquad\varphi(t,y)=e^{\int_{0}^{t}[2f(s,0)h(s,0)-h_{y}(s,0)]\ds}\int_{0}^{y}e^{-2\int_{0}^{z}f(t,x)\dx}\dz\\
-\int_{0}^{t}e^{\int_{0}^{u}[2f(s,0)h(s,0)-h_{y}(s,0)]\ds}h(u,0)\du\qquad\qquad\label{defvarphi}
\end{multline}
is strictly increasing in $y$ for each fixed $t$.
\end{thm}

\begin{proof}
By the discussion at the beginning of this section, it suffices to show, given the condition \eqref{eqlt5}, that the function $\varphi$ defined by \eqref{defvarphi} satisfies \eqref{eqlt0}, that $ \varphi(t,Y_{t})$ is a martingale,  and that $y\mapsto \varphi(0,y)$ is strictly increasing.
The last statement clearly follows from the fact that
\begin{align*}
\varphi(0,y)= \int_{0}^{y}e^{-2\int_{0}^{z}f(0,x)\dx}\dz.
\end{align*}
\par
Differentiating \eqref{defvarphi} in $y$ gives
\begin{align}\label{eqlt21}
\varphi_{y}(t,y)=e^{\int_{0}^{t}[2f(s,0)h(s,0)-h_{y}(s,0)]\ds}e^{-2\int_{0}^{y}f(t,x)\dx};
\end{align}
and hence
\[\varphi_{y}(t,y)=\varphi_{y}(t,0)e^{-2\int_{0}^{y}f(t,x)\dx},\]
which is \eqref{eqlt2}.
The equation \eqref{eqlt21} gives \eqref{eqlt3} via differentiating in $t$. From \eqref{eqlt21}, we also have \eqref{eqlt7} as well as the second equation in \eqref{eqlt0} by simple calculus. We now show the first equation in \eqref{eqlt0}.
\par
To this end, we rewrite $\varphi$ as
\begin{align*}
\varphi(t,y)&=\varphi_{y}(t,0)\int_{0}^{y}e^{-2\int_{0}^{z}f(t,x)\dx}\dz -\int_{0}^{t}\varphi_{y}(u,0)h(u,0)\du,
\end{align*}
which, by setting $y=0$ and then differentiation in $t$, yields
\begin{align}\label{eqlt22}
\varphi_{t}(t,0)+\varphi_{y}(t,0)h(t,0)=0.
\end{align}
\par
Notice \eqref{eqlt5} is equivalent to \eqref{eqlt6}, while the latter together with \eqref{eqlt2} implies \eqref{eqlt8}.
One can now easily deduce \eqref{eqlt4} from \eqref{eqlt8} and \eqref{eqlt7}. Comparing \eqref{eqlt4} and \eqref{eqlt3}, we see $\varphi_{ty}=\varphi_{y}(2fh-h_{y})$ which, together with the proved second equation in \eqref{eqlt0}, yields $\varphi_{ty}=-\varphi_{yy}h-\varphi_{y}h_{y}$. This implies that $\varphi_{t}+\varphi_{y}h$ does not depend on $y$, i.e.,
\[\varphi_{t}(t,y)+\varphi_{y}(t,y)h(t,y)\equiv\varphi_{t}(t,0)+\varphi_{y}(t,0)h(t,0).\]
But the right hand side is 0 by \eqref{eqlt22}, so is the left hand side, proving the first equation in \eqref{eqlt0}.
Finally, by \cite[Theorem 2]{BH06}, the BSDE \eqref{bsde1} admits a solution $(Y,Z)$ with
 $Z\in L^{2}_{\BF}([0,T])$.
It then follows from \eqref{itoequation} that
\[\dd \varphi(t,Y_{t})=\varphi_{y}(t,Y_{t})Z'_{t}\dd B_{t}.\]
It is easily seen from \eqref{eqlt21} that $\varphi_{y}$ is uniformly bounded;
so $ \varphi(t,Y_{t})$ is a martingale. The claim follows.
\end{proof}

\begin{remark}
Given any $h\in C^{0,2}([0,T]\times (0,\infty))$ with essentially bounded derivatives up to the second order, there exist (possibly multiple) $f\in C^{1,1}([0,T]\times\R)$ satisfying \eqref{eqlt5} by solving the resulting ODE. So
we have infinitely many ($z$-separable) law-invariant drivers.
\end{remark}

\subsection{An application to portfolio choice}
\noindent
Suppose the (self-financing) wealth process $Y$ of a small investor in a market consisting of a risk-free asset and multiple risky assets follows the SDE:
\begin{align} \label{newwealth1}
\dd Y_{t}=g(t,Y_{t},\sigma'_{t}\pi_{t})\dt+\pi'_{t}\sigma_{t}\dd B_{t},
\end{align}
with an initial wealth $Y_{0}>0$, where the vector $\pi_{t}$ denotes the dollar amounts invested in the risky assets at time $t$. Here we assume $\sigma_{t}^{-1}$ exists and is essentially bounded, in which case there is a BSDE with the driver $g$ associated with \eqref{newwealth1}.
\par
Consider the following dynamic portfolio choice problem under RDU:
\begin{align}\label{objective0}
\sup_{\pi(\cdot)} \quad & \int_0^{\infty} u(x)\dd\:\big(1-w(1-F_{Y_{T}}(x))\big),\\
\nonumber\textrm{subject to}\quad & (Y, \pi) \text{ follows \eqref{newwealth1}},\quad Y_{T}\geqslant 0.
\end{align}
As before, we first consider a static optimization problem in terms of $g$-expectation
\begin{align}\label{objective1}
\sup_{X} \quad & \int_0^{\infty} u(x)\dd\:\big(1-w(1-F_{X}(x))\big),\\
\nonumber\textrm{subject to}\quad & \BE[g]{X}=Y_{0},\quad X\geqslant 0.
\end{align}
If $X^{*}$ is an optimal solution of this problem, then the optimal portfolio $\pi^{*}$ of the problem \eqref{objective0} can be determined by the solution of the BSDE consisting of the SDE \eqref{newwealth1} and the terminal $X^{*}$.
\par
Taking the quantile $\mu^{-1}$ as the new decision variable, the problem \eqref{objective1} can be rewritten as
\begin{align}\label{objective2}
\sup_{X,\;\mu} \quad & \int_0^{1} u(\mu^{-1}(p))w'(1-p)\ddp,\\
\nonumber\textrm{subject to}\quad & \BE[g]{X}=Y_{0},\quad X\sim\mu, \quad \mu(0-)=0.
\end{align}
With the driver $g$ given in \citethm{THMlawin}, the above quantile optimization problem can be solved, leading to a complete solution to the problem \eqref{objective1}.
\begin{thm}
Suppose the driver in \eqref{newwealth1} is $g(t,y,z)\equiv f(t,y)\Vert z\Vert^{2}+h(t,y)$ with some functions $f\in C^{1,1}([0,T]\times\R)$ and $h\in C^{0,2}([0,T]\times\R)$ satisfying \eqref{eqlt5} and $f(T,\cdot)\leq0$. Furthermore we assume the probability distortion $w$ is concave and strictly increasing. Then the optimal solution of \eqref{objective1} is any random variable following the distribution $\mu^{*}$, where
$\mu^{*}\equiv \mu^{*}(\lambda)$ is uniquely determined by
\[ u'(x)w'(1-\mu^{*}(x))=\lambda \varphi_{y}(T,x),\]
with $\varphi$ defined by \eqref{defvarphi} and
the Lagrange multiplier $\lambda>0$ determined by $\BE{\varphi(T,\mu^{*}(\lambda))}=\varphi(0,Y_{0})$.
\end{thm}
\begin{proof}
We use the same notations used in \citethm{THMlawin} and its proof.
It follows from \citethm{THMlawin} that $g$ is a law-invariant driver. Moreover, $\varphi(0,\BE[g]{X})=\BE{\varphi(T,\mu)}$ for any $X\sim\mu$. Hence the problem \eqref{objective2} reduces to
\begin{align}\label{objective3}
\sup_{\mu} \quad & \int_0^{1} u(\mu^{-1}(p))w'(1-p)\ddp,\\
\nonumber\textrm{subject to}\quad & \BE{\varphi(T,\mu)}=\varphi(0,Y_{0}), \quad \mu(0-)=0.
\end{align}
Noting
\begin{align*}
\BE{\varphi(T,\mu)}=\int_{0}^{\infty}\varphi(T,x)\dd\mu(x)=\int_{0}^{1}\varphi(T,\mu^{-1}(p))\ddp,
\end{align*}
we can rewrite \eqref{objective3} in terms of $\mu^{-1}$:
\begin{align}\label{objective4}
\sup_{\mu^{-1}} \quad & \int_0^{1} u(\mu^{-1}(p))w'(1-p)\ddp,\\
\nonumber\textrm{subject to}\quad & \int_{0}^{1}\varphi(T,\mu^{-1}(p))\ddp=\varphi(0,Y_{0}), \quad \mu^{-1}(0+)\geq 0.
\end{align}
Because $\varphi_{y}(T,\cdot)\geq 0$ and $f(T,\cdot)\leq0$, we get from \eqref{eqlt0} that
\[\varphi_{yy}(T,y)=-2 \varphi_{y}(T,y)f(T,y)\geq 0,\]
implying that the mapping $y\mapsto\varphi(T,y)$ is convex. Since $\varphi_{y}(T,\cdot)\geq 0$,
the problem \eqref{objective4} is equivalent to
\begin{align}\label{objective5}
\sup_{\mu^{-1}} \quad & \int_0^{1} u(\mu^{-1}(p))w'(1-p)\ddp-\lambda \int_{0}^{1}\varphi(T,\mu^{-1}(p))\ddp,\\
\nonumber\textrm{subject to}\quad & \mu^{-1}(0+)\geq 0,
\end{align}
for some Lagrange multiplier $\lambda>0$.
\par
Because $y\mapsto\lambda\varphi(T,y)$ is convex and $u$ is strictly concave, the mapping
\[x\mapsto u(x)w'(1-p)-\lambda \varphi(T,x)\]
is strictly concave for each fixed $p\in(0,1)$. So its maximizer $x^{*}=x^{*}(p)$ is uniquely determined by the first-order condition
\[ u'(x^{*})w'(1-p)=\lambda \varphi_{y}(T,x^{*}).\]
Differentiating it in $p$ gives
\[ \big(u''(x^{*})w'(1-p)-\lambda \varphi_{yy}(T,x^{*})\big)\frac{\dd x^{*}}{\dd p}=u'(x^{*})w''(1-p),\]
from which it follows that $x^{*}(p)$ is increasing in $p$ because $w$ is concave. Let $(\mu^{*})^{-1}(p)\equiv x^{*}(p)$. Then it is a quantile function that maximizes the integrand of \eqref{objective5} point-wisely; so it is the optimal one. The proof concludes by noting
$u'(x)w'(1-\mu^{*}(x))=u'(x^{*})w'(1-p)=\lambda \varphi_{y}(T,x).$
\end{proof}

\begin{remark}
In the above result, for simplicity we have assumed that $w$ is concave. However, more general cases (e.g. when $w$ is inverse S-shaped) can be solved by the quantile optimization method developed recently.
We refer interested readers to \cite{HZ11,XZ13,HX16,XZ16,X16,XZZ16} for details.
\end{remark}

\section{Conclusions}

In this paper, we define the $g$-expectation of a distribution as the {\it infimum} of the $g$-expectations of the associated terminal random variables and discuss its interpretations in the context of portfolio choice. We could also change ``{\it infimum}'' to ``{\it supremum}'' in the definition to define a counterpart of this $g$-expectation.
A potential application of this alternative notion is in robust finance, in which one needs to set aside a sufficient initial fund to replicate (or hedge against) a distribution. We believe that a theory on this supremum version of the $g$-expectation of distributions can be developed in parallel to that of the current, infimum version.

We end by noting that this paper raises more questions than giving solutions. The problem of BSDEs and the related $g$-expectations with distributions is a new research direction, and this paper gives very limited, preliminary results.
There are many interesting problems to be explored, such as the uniqueness in law of BSDEs with terminal distributions, a complete characterization of law-invariant $g$-expectation, and axiomatic characterization of $g$-expectation of distributions for both the infimum and supremum versions.



\end{document}